\documentclass{article}
\usepackage{graphicx} 
\usepackage[utf8]{inputenc}
\usepackage{fullpage}
\usepackage{amsfonts}
\usepackage[Symbol]{upgreek}
\usepackage[normalem]{ulem}
\usepackage[dvipsnames]{xcolor}
\usepackage{geometry} 
\usepackage{graphicx}
\usepackage{bbm}
\usepackage{dsfont}
\usepackage{cancel}
\usepackage{enumerate, mathtools}
\usepackage{tikz}  
\usepackage{amsmath}
\usepackage{cite}
\usepackage{relsize}
\usepackage{amsthm} 
\usepackage{pgfplots}
\usepgfplotslibrary{fillbetween}
\usepackage{blindtext}
\usepackage{subfiles}
\usepackage{comment}
\usepackage{mathrsfs}
\usepackage{enumitem}
\usepackage[colorinlistoftodos]{todonotes}
\usepackage{amssymb}
\usepackage[colorlinks]{hyperref}
\usepackage[colorinlistoftodos]{todonotes}
\usetikzlibrary{3d,spy}
\usepackage{verbatim}
\pgfplotsset{width=10cm,compat=1.15}
\usetikzlibrary{positioning, calc, matrix, arrows.meta}
\tikzset{
    every matrix/.style={matrix of math nodes},hp/.style={column 3/.style={anchor={base east}}},Hp/.style={column 4/.style={anchor={base east}}},>=Stealth
}
\definecolor{a0}{rgb}{0.61, 0.77, 0.89}
\definecolor{a1}{rgb}{1.0, 0.75, 0.0}
\definecolor{a2}{rgb}{1.0, 0.74, 0.53}
\makeatletter
\def\smallunderbrace#1{\mathop{\vtop{\m@th\ialign{##\crcr
   $\hfil\displaystyle{#1}\hfil$\crcr
   \noalign{\kern3\p@\nointerlineskip}%
   \tiny\upbracefill\crcr\noalign{\kern3\p@}}}}\limits}
\makeatother
 
\setlist[enumerate]{itemsep=2mm}
\usepackage[utf8]{inputenc}
\usepackage{color}
\usepackage{xcolor}
\usepackage{makecell}
\usepackage{ragged2e}
\usepackage{wrapfig}

\usepackage[english]{babel}
\theoremstyle{remark}

\newtheorem{theorem}{Theorem}[section]
\newtheorem{prop}[theorem]{Proposition}
\newtheorem{theo}{Theorem}[section]

\newtheorem{coro}[theo]{Corollary}

\newtheorem{lemma}[theorem]{Lemma}
\theoremstyle{definition}
\newtheorem{defn}[theorem]{Definition}
\newtheorem*{theorem*}{Theorem}
\newtheorem*{prop*}{Proposition}
\newtheorem*{corollary*}{Corollary}
\newtheorem*{lemma*}{Lemma}
\newtheorem*{defn*}{Definition}
\newtheorem*{remark*}{Remark}

\newcommand{\vol}{\textnormal{vol}}

\newcommand*\bigcdot{\mathpalette\bigcdot@{.5}}

\newcommand{\var}{\text{var}}

\newcommand\csout{\bgroup\markoverwith{\textcolor{red}{\rule[0.5ex]{2pt}{0.7pt}}}\ULon}

\DeclareSymbolFont{euler}{U}{eur}{m}{n}
\DeclareMathSymbol \uppi \mathalpha {euler} {"19}

\title{On Convex Functions of Gaussian Variables}
\author{Maite Fern\'andez-Unzueta, James Melbourne, and Gerardo Palafox-Castillo}
\date{}

\begin{document}

\maketitle

\begin{abstract}
    We investigate a convexity properties for normalized log moment generating function continuing a recent investigation of Chen of convex images of Gaussians. We show that any variable satisfying a ``Ehrhard-like'' property for its distribution function has a strictly convex normalized log moment generating function, unless the variable is Gaussian, in which case affine-ness is achieved.  Moreover we characterize variables that satisfy the Ehrhard-like property as the convex images of Gaussians.  As applications, we derive sharp comparisons between R\'enyi divergences for a Gaussian and a strongly log-concave variable, and characterize the equality case.  We also demonstrate essentially optimal concentration bounds for the sequence of conic intrinsic volumes associated to convex cone and we obtain a reversal of McMullen's inequality between the sum of the (Euclidean) intrinsic volumes associated to a convex body and the body's mean width that generalizes and sharpens a result of Alonso-Hernandez-Yepes.
\end{abstract}
\section{Introduction}

We will investigate random variables $X$ such that $\Lambda_X(\lambda) \coloneqq \frac 1 \lambda \log \mathbb{E} e^{\lambda X}$ is convex in $\lambda$.  This class of random variables, as recently proven by Chen \cite{chen2023gaussian}, includes all convex images of $n$-dimensional Gaussian random vectors. We will give an elementary proof that a random variable $X$ has such a convexity property for its moment generating function if its distribution $F$ such that $\Phi^{-1} \circ F$ is concave, thus ``generalizing'' the result of Chen.  

\begin{theo} \label{thm: main theorem of intro}
    For $X$ a random variable such that 
    \[
        t \mapsto \Phi^{-1}( \mathbb{P} [X < t])
    \]
    is concave, then $\Lambda : \mathbb{R} \to \mathbb{R} \cup \{\infty\}$ defined by
    \[
        \Lambda_X(p) = \frac 1 p \log \mathbb{E} e^{pX}
    \]
    is convex, and moreover is strictly convex unless $X$ is Gaussian.
\end{theo}
We will see that both our result and Chen's are contained in the one dimensional version of the result for convex images of Gaussian variables. That is a variable $X$ satisfies $t \mapsto \Phi^{-1}(\mathbb{P}[ X > t])$ concave if and only if $X \sim f(Z)$ for $Z$ a standard normal random variable on $\mathbb{R}$ and $f$ a convex function. 

As an application we consider comparability of R\'enyi divergences from Gaussianity for strongly log-concave random vectors.

 \begin{theo} \label{thm: renyi divergence comparison intro}
  For $0<\alpha < \beta < \infty$, $X$ a random variable with a log-concave density with respect to a Gaussian measure $\gamma$, then $Z \sim \gamma$ implies that $\alpha$ and $\beta$ R\'enyi divergences are equivalent,
     \[
         D_\beta(X||Z) \leq \frac{\beta}{\alpha} D_\alpha(X||Z).
     \]
     {\color{black}with equality  iff $X$ is a translation of $Z$}.
 \end{theo}

The strict convexity of Theorem \ref{thm: main theorem of intro} for non-Gaussian variables is combined with a characterization of affine maps as the convex functions that preserve Gaussianity (proven in \ref{thm: convex functions preserving Gaussianity are affine}) to  obtain the equality condition in Theorem \ref{thm: renyi divergence comparison intro}.

We will also leverage the deviation inequalities obtained through the convexity of the moment generating functions to derive new concentration bounds for sequences of conic intrinsic volume.  More explicitly we show that conic-intrinsic volume random variables are sub-Gaussian, thus obtaining an essentially optimal improvement the deviation bounds obtained in \cite{amelunxen2014living, mccoy2014steiner}.

\begin{theo}
    For $C$ a convex cone,  $V_C$ the associated intrinsic volume random variable and $\delta(C)$ its statistical dimension, 
    \[
        \mathbb{P}(|V_C - \delta(C)| \geq t ) \leq 2 e^{- \frac{t^2}{2 \sigma^2}}
    \]
    with sub-Gaussian variance proxy $\sigma^2 \coloneqq\var(V_C) +  2 \max \{ \delta(C), n - \delta(C)\}$.
\end{theo}
When $C = [0,\infty)^n$ the associated intrinsic volume random variable is a Binomial$(1/2,n)$, known to have sub-Gaussian variance proxy $\frac n 4$.  Here $\var(V_C) + 2 \max \{ \delta(C), n - \delta(C) \} = \frac{5n}{4}$.  Thus up to the constant $5$ these bounds cannot be improved in general.

We also obtain a reversal of McMullen's classical inequality \cite{mcmullen1991inequalities} between the sum of the intrinsic volumes and the mean width\footnote{More explicitly $W(K) \leq e^{V_1(K)}$ for $W(K)$ the sum of the intrinsic volumes associated to $K$ and $V_1(K)$ its intrinsic volume of order $1$, a constant multiple of the mean width.}, improving \cite[Theorem 1.2]{alonso2021further}.
\begin{theo} \label{thm: intro reverse mcmullen}
    For $Z \sim N(0,I_n)$, a convex set $K \subseteq B_r(0)$, the Euclidean unit ball in $\mathbb{R}^n$ of radius $r$, and $f_K(z) = \sup_{y \in \sqrt{2 \pi} K} \left( \langle z , y \rangle - \frac{|y|^2}{2} \right)$,
    \[
        \log \sum_{i=0}^n V_i(K) \geq \frac{\var( f_K(Z))}{2} +V_1(K) - \frac{r^2}{2},
    \]
    where $V_i(K)$ denotes the $i$-th intrinsic volume associated to $K$ and $\var( \cdot )$ denotes the usual variance of a random variable.
\end{theo}

Let us describe the remaining contents of the paper.  In Section \ref{sec: Prelims} we recall some definition and requisite results of Ehrhard.  Section \ref{sec: convexity of MGF} contains the proof of our main technical result, which is obtained through an elementary ``line-crossing'' argument.  Some technical, but elementary facts are derived in an appendix.  In Section \ref{sec: deviation bounds and characterization} we derive through standard methods some deviation inequalities from the moment generating function convexity that prove useful in our applications.  We also show that our study in the general case is equivalent to the case that our variable $X$ is the convex image of a standard Gaussian random variable.  In Section \ref{sec: Renyi divergence bounds} we derive sharp bound on the difference between the $\alpha$ and $\beta$-R\'enyi entropies between two variables $X$ and $Z$, where $Z$ is a standard Gaussian random vector and $X$ is a strongly log-concave random vector. In Section \ref{sec: conic intrinsic volumes} we derive sub-Gaussian deviation bounds for the conic intrinsic volume random variables and in Section \ref{sec: Wills} we obtain the sharpening of Alonso, Hern\'andez, Yepes \cite{alonso2021further} reversal of McMullen's inequality on the sum of intrinsic volumes of a convex body and its mean width \cite{mcmullen1991inequalities}. The appendix contains an elementary algebraic lemma, that allows the invocation of a ``line crossing argument'' that underlies our main technical result, Theorem \ref{thm: main theorem of intro}.

\section{Preliminaries} \label{sec: Prelims}
We let $\gamma_n$ denote the standard $n$-dimensional Gaussian measure, that is 
\[
    \gamma_n(A) = \int_A e^{-|x|^2/2} \frac{dx}{(2\pi)^{\frac n 2}},
\]
an $\mathbb{R}^n$-valued random variable $X$ is a standard Gaussian if $\mathbb{P}(X \in A) = \gamma_n(A)$.  We denote 
\[
    \Phi(t) \coloneqq \gamma_1((-\infty,t)).
\]
We denote by $\Phi^{-1}$ the inverse function of $\Phi$ and call an $\mathbb{R}^n$ valued random variable $X$ Gaussian if it is the affine image of a standard $n$-dimensional Gaussian. 
\begin{theo}[Ehrhard \cite{ehrhard1983symetrisation}]\label{thm: Borell Ehrhardt}
    For convex sets $A, B$ of $\mathbb{R}^n$ and $\gamma$ a Gaussian measure
    \begin{align} \label{eq: Borell-Ehrhardt inequality}
    \gamma((1-t) A + t B) \geq \Phi \left( (1-t) \Phi^{-1}(\gamma(A)) + t \Phi^{-1}(\gamma(B)) \right)
    \end{align}
\end{theo}
Note that stronger versions of Ehrhard's inequality have been proven, Lata{\l}a \cite{Lat96} extended the inequality to the setting that $A$ is convex and $B$ is Borel, and Borell \cite{Bor03} proved that the inequality holds for all Borel $A$ and $B$.
We take the convention that $\frac 1 p \log \mathbb{E} e^{pf(X)} = \mathbb{E} f(X)$ when $p =0$ in the following main theorem.  Note that when $\mathbb{E} e^{p f(X)}$ is integrable for some $ p >0$ then 
$\lim_{p \to 0} \frac 1 p \log \mathbb{E} e^{pf(X)} = \mathbb{E} f(X)$.



\section{Convexity of Normalized log Moment Generating Functions} \label{sec: convexity of MGF}

The main result of this section is the proof of Theorem \ref{thm: main theorem of intro}, restated below.
\begin{theo} \label{thm: main generalized}
    For $X$ a random variable such that 
    \[
        t \mapsto \Phi^{-1}( \mathbb{P} [X > t])
    \]
    is concave, then $\Lambda : \mathbb{R} \to \mathbb{R} \cup \{\infty\}$ defined by
    \[
        \Lambda_X(p) = \frac 1 p \log \mathbb{E} e^{pX}
    \]
    is concave, and moreover is strictly concave unless $X$ is Gaussian.
\end{theo}

Note that the above is indeed equivalent to Theorem \ref{thm: main theorem of intro}.  For instance if the above holds then for $X$ such that $\Phi^{-1}(\mathbb{P}([X < t])$ is concave, $t \mapsto \Phi^{-1}(\mathbb{P}[ -X > -t])$ is concave so that $\Lambda_{-X}(p)$ is concave, and hence $- \Lambda_{-X}(p) = \Lambda_{X}(-p)$ is convex and Theorem \ref{thm: main theorem of intro} follows.  The converse is similar.

\begin{proof}
By the monotonicity of $L_p$-norms against a probability measure, the $\{ \Lambda < \infty \}$ is a ray. We consider first $p_0 <  p_1 \in \mathbb{R} - \{0\}$ such that $\mathbb{E} e^{p_1 X} < \infty$.   Define $X_o \coloneq\sigma Z + \mu$ for $Z$ a standard one dimensional Gaussian, with
    \[
        \mu \coloneqq \frac{p_1 \Lambda(p_0) - p_0 \Lambda(p_1)}{p_1 - p_0} \hspace{8mm} \sigma^2 \coloneqq 2 \frac{ \Lambda(p_1) - \Lambda(p_0)}{ p_1 - p_0},
    \]
    chosen such that $\Lambda_o \coloneqq \frac 1 p \log \mathbb{E} e^{p X_o}$ satisfies 
    \begin{align} \label{eq: big Lambda equality}
        \Lambda(p_i) = \Lambda_o(p_i).
    \end{align}
    
 Observe that $\Lambda_{o}(p) = \frac{\sigma^2}{2} p + \mu$ is affine and
\[
    \Phi^{-1}(\mathbb{P}[X > t]) - \Phi^{-1}(\mathbb{P}[ X_o > t ]) = \Phi^{-1}(\mathbb{P}[X > t]) - \frac{t - \mu}{\sigma}
\]
is thus concave and hence it follows that $\{ \phi >0 \}$ where
\[
    \phi(t) \coloneqq \mathbb{P}[X > t] - \mathbb{P}[X_o > t]
\]
is an interval, with endpoints we denote by $x_0$ and $x_1$. Applying  Fubini-Tonelli to the expression $e^{x} = \int_{\mathbb{R}}e^y \mathbbm{1}_{\{x > y\}} dy $, for $p \neq 0$,
\[
   \frac{1}{p} \left(  \mathbb{E} e^{p X} - \mathbb{E}e^{p X_0} \right) 
        = 
            \int_{\mathbb{R}} e^{pt}\phi(t) dt.
\]
Fix $0 \neq p \in (p_0, p_1)$ and using \eqref{eq: big Lambda equality}, for constants $c_0$ and $c_1$, we have
\[
    \int_{\mathbb{R}} e^{pt}\phi(t) dt = \int_{\mathbb{R}} \Psi(t)  \phi(t) dt.
\]
where $\Psi(t) \coloneqq \Psi_{p,c} (t) \coloneqq e^{pt} + c_0 e^{p_0 t} + c_1 e^{p_1 t} $ and $c_0$ and $c_1$ chosen, through the linear independence of the vectors $(e^{p_k x_i}) \in \mathbb{R}^2$ proven in Lemma \ref{lem: exponential lin ind} such that $\Psi(x_0) = \Psi(x_1) = 0$.
 Since $\Psi$ is a sum of three exponential functions with two zeros, and $p_0 < p < p_1$ it follows that $c_0$ and $c_1$ are negative.  In particular $\Psi$ is negative for large values of $t$.  Thus, again by Lemma \ref{lem: exponential algebra stuff}, $\{ \Psi > 0 \} = (x_0,x_1)$, and hence $\Psi \phi \geq 0$ on $\mathbb{R}$.  Thus
\[
   \frac{1}{p} \left(  \mathbb{E} e^{p X} - \mathbb{E}e^{p X_0} \right) = \int_{\mathbb{R}} \Psi \phi dt \geq 0,
\]
with equality iff $\phi = 0$ almost surely which is iff $X$ is a Gaussian.
After writing $p = (1-\lambda) p_0 + \lambda p_1$, this gives
\[
    \Lambda(p) \geq \Lambda_{o} ((1-t) p_0 + t p_1) = (1-t)\Lambda_{o}(p_0) + t \Lambda_{o}(p_1) = (1-t)\Lambda(p_0) + t \Lambda(p_1).
\]
with equality iff $X$ is Gaussian.
In the case that $\mathbb{E} f(X)  = \infty$ there is nothing else to prove.  When $\mathbb{E} X$ is finite, {\color{black} the proof follows by the continuity of $\Lambda$} since $\Lambda$ is strictly convex at every point of away from zero.
\end{proof}


The following lemma is an immediate consequence of the Ehrhard inequality.  We include the proof for convenience.
\begin{lemma}[Ehrhard \cite{ehrhard1984inegalites}] \label{lemma: ehrhard concave}
    For a concave  function $f$ and a Gaussian measure $\gamma$ 
    the function
    \[
        \Phi^{-1}( \gamma\{ f > t \})
    \]
    is concave on $\mathbb{R}$ as a map to $\mathbb{R} \cup \{-\infty\}$.  For $g$ convex { $\Phi^{-1}( \gamma \{ g < t \})$}, is concave.
\end{lemma}

\begin{proof}
    For  { $t_i \in \mathbb{R}$} and $\lambda \in (0,1)$, the concavity of $f$ implies that
    \[
        \{ f > (1-\lambda) t_0 + \lambda t_1 \} \supseteq (1- \lambda ) \{ f > t_0 \} + \lambda \{ f > t_1 \}.
    \]
    Taking the $\gamma$ measure of the two sets and applying Ehrhardt's inequality to the convex set $\{ f > t_i \}$, 
    \begin{align*}
        \gamma( \{ f > (1-\lambda) t_0 + \lambda t_1 \})
            &\geq
                \gamma( (1-\lambda) \{ f > t_0 \} + \lambda \{ f > t_1 \})
                    \\
            &\geq
                \Phi( (1-\lambda) \Phi^{-1}(\gamma\{f > t_0\}) + \lambda \Phi^{-1}(\gamma\{f > t_1 \}) ).
    \end{align*}
    Taking $\Phi^{-1}$ of both sides completes the proof. 
    If $g$ is a convex function, then $\Psi(t) \coloneqq\Phi^{-1}(\gamma\{ g < t \}) = \Phi^{-1}( \gamma \{ -g > -t \} )$.  By the previous result $\Psi(-t)$ is concave and hence $\Psi$ is as well.
\end{proof}

When $X = f(Z)$  for $f:\mathbb{R}^n \to \mathbb{R}$ convex and $Z$ a Gaussian then the application of  Lemma \ref{lemma: ehrhard concave},  shows that Theorem \ref{thm: main generalized} gives a strenghtening and a {\it formal} generalization of the following recent result of Chen.

\begin{theo}[Chen \cite{chen2023gaussian}]
For a convex function $f$ and $Z$ a standard Gaussian random vector,
\[
    \Lambda(p) = \frac 1 p \log \mathbb{E} e^{p f(Z)}
\]
is convex.
\end{theo}

\begin{proof}
    Taking $X = f(Z)$, we have by Theorem \ref{lemma: ehrhard concave} that $\mathbb{P}( X < t )$ is concave, and hence applying Theorem \ref{thm: main generalized} to $X$, the result follows.
\end{proof}

Moreover, Theorem \ref{thm: main generalized} strengthens the above to $\Lambda$ is strictly convex unless $f(Z)$ is Gaussian.  In what follows we will characterize the affine functions $f: \mathbb{R}^n \to \mathbb{R}$ as the convex functions that preserve Gaussianity. A fact surely known, but not one we found in the literature. 

\begin{theo}[Shenfield - van Handel \cite{shenfeld2018equality}] \label{thm: equality in Borell Ehrhard}
    Let $A, B \subseteq \mathbb{R}^n$ be closed sets with $\gamma_n (A), \gamma_n(B) \in (0,1)$ then equality holds in Borell-Ehrhard if and only if
    \[
        A = \{ x \in \mathbb{R}^n : \langle a, x \rangle \geq b\}, \ \ \ B = \{ x \in \mathbb{R}^n : \langle a , x \rangle \geq c \}
    \]
    for some $a \in \mathbb{R}^n$ and $b,c \in \mathbb{R}$ or
    $A$ and $B$ are convex and $A = B$.
\end{theo}

\begin{theo} \label{thm: convex functions preserving Gaussianity are affine}
    If $Z$ is a Gaussian vector and $f$ is a convex function such that $f(Z)$ is a Gaussian random variable then $f$ is affine.
\end{theo}

\begin{proof}
    By a change of variable it suffices to consider $Z$ and $f(Z)$ standard Gaussians.  Thus $\Psi(t) \coloneqq \Phi^{-1}( \mathbb{P}(f(Z) \leq t)) = t$. Hence for $t,s \in \mathbb{R}$,
    by convexity $\{  f \leq (1-\lambda) t + \lambda s \} \supseteq (1-\lambda) \{f \leq t \} + \lambda \{f \leq s\}$ hence
    \begin{align*}
        (1-\lambda) t + \lambda s
            &\geq \Phi^{-1}(\gamma_n((1-\lambda)\{f \leq t \} +  \lambda \{ f \leq s\}))
                \\
            &\geq (1-\lambda)\Phi^{-1}(\gamma_n\{f \leq t \}) + \lambda \Phi^{-1}(\gamma_n \{f \leq s \})
                \\
            &=
                (1-\lambda) t + \lambda s
    \end{align*}
    Thus, it follows that   we  have  equality  in the Borell-Ehrhard inequality and for $t \neq s$, $\gamma \{ f \leq s\} \neq \gamma\{f \leq t\}$, and hence $\{f \leq s \} \neq \{ f \leq t \}$ and hence it follows that we must have that there exists a unit vector $a$
    \[
        \{f \leq t\} = \{x\in \mathbb{R}^n : \langle a,x\rangle \leq \lambda(t)\}
    \]
    However, since $t = \gamma_n \{ f \leq t\} = \gamma_n \{x \in \mathbb{R}^n : \langle x, a \rangle \leq \lambda(t) \} = \lambda(t)$.  Thus $f$ has the same sublevel sets as $x \mapsto \langle a,x\rangle$, and hence we have $f(x) = \langle a, x\rangle$, completing the proof.
    
\end{proof}

Thus in the context of a Gaussian random variable composed with a convex function we have the following.
\begin{coro} \label{cor: convex function of Gaussian}
    For $Z$ a Gaussian vector in $\mathbb{R}^n$ and $f$ a convex function on $\mathbb{R}^n$, then
    \[
        \Lambda(p) = \frac 1 p \log \left( \mathbb{E}e^{p f(Z)}\right)
    \]
    is convex, and strictly so unless $f$, and hence $\Lambda$, are affine.

    \begin{proof}
        Theorem \ref{thm: main generalized} already gives that $\Lambda$ is convex, and strictly so unless $f(Z)$ is Gaussian. By Theorem \ref{thm: convex functions preserving Gaussianity are affine}, $f(Z)$ is non-Gaussian unless $f$ is affine. 
    \end{proof}
\end{coro}
 Again, applying the result to $f = -g$ for $g$ concave implies that $\Lambda$ is concave when $g$ is.

\section{Convex Images of a Standard Normal} \label{sec: deviation bounds and characterization}
Here we will derive deviation inequalities as consequences of convexity properties of the moment generating functions.  Following this, we will characterize $\{\mu, f\}$ such that $\lambda \mapsto \Phi^{-1}(\mu \{ f \leq \lambda\})$ is concave by the property that the pushforward of $\mu$ by $f$, $f \# \mu$ is the convex image of a one dimensional standard Gaussian.  But first we will derive tail bounds for variables $X$ such that $\Lambda_X$ is convex.

\subsection{General Deviation Bounds}
As mentioned the convexity of $\Lambda_{f(Z)}$ for $f$ convex and $Z$ Gaussian was proven \cite{chen2023gaussian}, where standard techniques give one sided deviation bounds. For the convenience of the reader we sketch the derivation of the inequality we will use.  

For $X$ such that $\frac 1 \lambda \log \mathbb{E} e^{\lambda X}$ is convex and $\mathbb{E} X^2 < \infty $ we outline the fact that
\begin{align} \label{eq: MGF bound under convexity}
    \mathbb{E} e^{\lambda X} \leq e^{\frac{\lambda^2}{2} \var(X) + \lambda \mathbb{E} X }
\end{align}
for $\lambda \leq 0$.  It suffices to prove the result when $\mathbb{E}X = 0$, in which case $\Lambda(0) = 0$, and we have by the increasingness of $\Lambda'$
\begin{align} \label{eq: FTC on Lambda}
    \Lambda(\lambda) = \int_0^{\lambda} \Lambda'(t) dt \geq \lambda \Lambda'(0)
\end{align}
In this case
\[
    \frac{\Lambda(\lambda)}{\lambda} = \frac{1}{\lambda^2} \log \left( 1 + \left( \mathbb{E}e^{\lambda X} -1 \right) \right) = \frac 1 {\lambda^2} \left( \sum_{n=2}^\infty \frac{\lambda^n \mathbb{E}X^n}{n!} + o(\lambda^2) \right) = \frac{\var(X)}{2} + o(1) 
\]
Thus $\Lambda'(0) = \var(X)/2$, inserting into \eqref{eq: FTC on Lambda}, and rearranging the result follows.  Note that by Chernoff's method we arrive at the following.

\begin{theo}[Chen \cite{chen2023gaussian}]
    For area valued random variable $X$, with $\mathbb{E} X^2 < \infty$ such that $\lambda \mapsto \frac 1 \lambda \log \mathbb{E} e^{\lambda X}$ is convex on $(-\infty, 0]$, we have
    \[
        \mathbb{E} e^{\lambda X} \leq \exp  \left \{ \frac{\lambda^2}{2} \var(X) + \lambda \mathbb{E} X \right \}
    \]
    and
    \[
        \mathbb{P}( X \leq \mathbb{E}X - t) \leq \exp \left\{ - \frac{t^2}{2  \var(X)} \right \}
    \]
\end{theo}

\begin{proof}
    The inequality for the moment generating functions is already outlined above.  The following argument for the tail bounds is standard. For $\lambda \leq 0$, using Markov's inequality and then the moment generating function bound previously derived we have
    \begin{align*}
        \mathbb{P}( X \leq \mathbb{E}X - t)
            &=
                \mathbb{P} \left( e^{\lambda (X- \mathbb{E}X)} \geq e^{-\lambda t } \right)
                    \\
            &\leq 
                e^{\lambda t}\mathbb{E}e^{\lambda (X- \mathbb{E}X) }
                    \\
            &\leq 
                \exp \left\{ \frac{\lambda^2}{2} \var(X) + \lambda t  \right\}.
    \end{align*}
    Taking $\lambda = \frac{-t}{\var(X)}$ completes the proof.
\end{proof}



\subsection{Characterization of $\Phi^{-1}$-concavity of sublevel sets}
In Valettas \cite{valettas2019tightness} it is mentioned that ``the property that $\Phi^{-1}(\gamma_n \{ f \leq t \} )$ is concave is shared by other significant
 distributions at the cost of fairly restricting the class of convex functions''.  For example, borrowing from \cite{paouris2018gaussian}  and \cite{valettas2019tightness}, we have the following.
 
 \begin{theo}
For $X$, independent $\mathbb{R}^n$-valued random variable with distribibution function $\phi_X$
 \[
    \phi_X(x) = \exp \left\{ - \sum_{k=1}^n x_k \right \} \mathbbm{1}_{(0,\infty)^n},
 \]
 and $f$ is a coordinate increasing convex function then
 \[
    \mathbb{P}(f(X) \leq \mathbb{E}f(X) - t) \leq \exp \left\{ \frac{-t^2}{ 2\var(f(X))} \right\}
 \]
\end{theo}

\begin{proof}
   Taking $F: \mathbb{R}^{2n} \to \mathbb{R}$ by $F(x,y) = f(x_1^2 + y_1^2, \dots, x_n^2 +y_n^2)$, we have $F$ to be convex and since the sum of two independent Gaussians is exponential, the proof concludes.
\end{proof}

The above is a straightforward adaptation of Theorem \ref{thm: main generalized}. In what follows we will show that all examples of measures $\{\mu, f\}$ such that $t \mapsto \Phi^{-1} (\mu \{ f \leq t \})$ is concave are straightforward adaptations of Theorem \ref{thm: main generalized}.   More explicitly, we will show that the class measures realized as the pushforward of a probability measure, whose distribution function composed with $\Phi^{-1}$ is concave, are characterized as the convex images of a Gaussian. To this end, we develop some elementary ideas from the optimal transport.  

In the theory of optimal transport on metric measure space $E$, \cite{Vil09:book} one studies
\[
    \inf_{\pi \in \Pi(\mu,\nu)} \int_{E\times E} c(x,y) d \pi
\]
where $\Pi(\mu, \nu)$ is the space of all measures $\pi$ on $E \times E$ that couple two probability measures $\mu$ and $\nu$ in the sense that $\pi(A \times E ) = \mu(A)$ and $\pi(E \times B) = \nu(B)$ for arbitrary measurable $A,B \subseteq E$, and $c: E \times E \to \mathbb{R} \cup \{\infty\}$ is a (typically lower semi-continuous) ``cost function''.  The subject is rich and has deep connections across mathematics that we will not attempt to do justice. With this said, the case that $E = \mathbb{R}$ and $c(x,y) = |x-y|^2$ for measures $\mu$ and $\nu$ with convexly supported densities with respect to the Lebesgue measure is remarkably simple.  Writing 
\[
    F(x) = \mu(-\infty, x] \  \ \hbox{ and }  \ \ G(x) = \nu (-\infty, x]
\]
Suppose for an increasing $T$, the pushfoward of $\mu$ under $\mu$, $T\# \mu$ satisfies $T\# \mu = \nu$, then
\[
    G(x) = \nu (-\infty,x] = T\#\mu (-\infty,x] = \mu (-\infty, T^{-1}(x)] = F(T^{-1}(x))
\]
Hence $F^{-1} \circ G = T^{-1}$ and hence $T = G^{-1} \circ F$.  That is $(G^{-1} \circ F)\#\mu = \nu$.  Written in terms of random variables $X \sim \mu$ and $Y \sim \nu$, 
\[
    G^{-1}(F(X)) = Y
\]
When we write that $\Phi^{-1}( \mathbb{P}( X \leq t))$ is concave, taking $F(x) = \mathbb{P}(X \leq x)$, we are writing that $\Phi^{-1} \circ F$ the transport map taking $X$ to a standard Gaussian $Z$ is concave.  Note that the inverse function of an increasing concave function\footnote{Direct computation $(f^{-1})'(x) = \frac 1 {f'(f^{-1}(x))} $ which is an increasing function as $f^{-1}$ is increasing, $f'$ is decreasing and $1/x$ is decreasing. } is an increasing convex function.  Thus in particular, $X \sim F^{-1}\circ \Phi(Z)$, where $F^{-1}\circ \Phi$ is a convex function.  In summary we have the following characterization.

\begin{theo}
For a function $f$ and a random variable $X$ with distribution $\mu$, then
\[
    \Phi^{-1}(\mu \{ f \leq t \})
\]
is concave if and only if $f(X) \sim \varphi(Z)$ for $\varphi$ convex and $Z$ a standard $\mathbb{R}$-valued Gaussian random variable.
\end{theo}

Let us emphasize that this only characterizes the concavity of $\Phi^{-1}$ composed with the measure of sublevel sets of a function, a sufficient condition to guarantee that $p \mapsto \frac 1 p \log \mathbb{E} e^{p f(X)}$ is convex for $X \sim \mu$. However, the convexity of the normalized log-moment generating function holds for many other distributions that are not convex images of Gaussian random variables.  For instance, if $X \sim$ Poisson($p$) then 
\[
    \frac 1 \lambda \log \mathbb{E}e^{\lambda X} = \frac{p (e^\lambda - 1)}{\lambda}
\]
is a convex function function, but $X$ is plainly not the convex image of a Gaussian random variable.

\section{R\'enyi Divergence Comparisons} \label{sec: Renyi divergence bounds}

For $\alpha \in (0,1) \cup (1,\infty)$, the R\'enyi entropy of a random variable $X$ with density $f$ with respect to the Lebesgue measure is defined by
\[
   h_\alpha(X) \coloneqq h_\alpha(f) \coloneqq \frac{\log  \int f^\alpha dx}{ 1 - \alpha},
\]
Completing the definition for $\alpha \in \{ 0, 1 , \infty\}$, by
\begin{align*}
    h_0(X) &\coloneqq \log \mu \{ f > 0\}
        \\
    h_1(X) &\coloneqq - \int f \log f d\mu
        \\
    h_\infty &\coloneqq - \log \|f\|_\infty,
\end{align*}
where the $\|f\|_\infty$ denotes the essential supremum of $f$ taken with respect to $\mu$, we see that the R\'enyi entropy interpolates, notions of volume when $\alpha =0$, the classical Shannon entropy when $\alpha =1$, the restriction of the $L_2$ Hilbert space  to probability densities when $\alpha =2$ and notions of concentration when $\alpha = \infty$.  As such it has provided a bridge between disciplines. Perhaps the most celebrated such connection is that of the Brunn-Minkowski inequality and the Shannon-Stam Entropy Power inequality \cite{CC84}, explicitly connected through Sharp Young inequality in \cite{DCT91}.  These connections go much further, we direct the interested reader to more work \cite{CC09, madiman2017forward, van2014renyi, LCCV18, MMX17:1}.  Observing that
\[
    h_\alpha(X) = - \log \left(\int f^{\alpha-1} d \nu \right)^{\frac 1 {\alpha-1}}
\]
where $\nu$ is the probability measure induced by the density $f$ against $dx$, we see through an application of Jensen's inequality that the $\alpha \mapsto h_\alpha(X)$ is a non-increasing function, so that $h_\alpha(X) \leq h_\beta(X)$ for $\alpha \geq \beta$. More interestingly, this inequality can be reversed up to sharp constants for a reasonably large and important classes of random variables.  In particular (see \cite{FMW16} for details), for $X \sim f$ with $f$ log-concave in the sense that $f = e^{-V}$ for a convex function $V : \mathbb{R}^n \to \mathbb{R} \cup \{ \infty \},$ then for $\alpha \geq \beta$
\[
    h_\alpha(X) \leq h_\beta(X) \leq  h_\alpha(X) + c_{\alpha,\beta}
\]
where $c_{\alpha, \beta} = h_\beta(Y) - h_\alpha(Y)$ where $Y \sim g(x) = \mathbbm{1}_{(0,\infty)^n}(x) e^{- \sum_{i=1}^n x_i}$.  In short, among log-concave distributions, the exponential delivers the largest difference between R\'enyi entropies.  This is echoed to some degree in the discrete setting, when the Lebesgue measure is replaced by the counting measure on $\mathbb{Z}$ \cite{melbourne2020reversal,melbourne2023discrete}.
In what follows we establish analogous results for what would be the entropy against the Gaussian measure, to this end let us recall the $\alpha$-R\'enyi divergence between random variables $X \sim \mu$ and $Y \sim \nu$ with Radon-Nikodym derivative $\frac{d\mu}{d\nu}$ is  defined by 
\[
    D_\alpha(X||Y) \coloneqq (\alpha -1 )^{-1} \log \int \left( \frac{d\mu}{d\nu} \right)^\alpha d\nu = (\alpha - 1)^{-1} \log \int \left( \frac{d\mu}{d\nu} \right)^{\alpha -1} d\mu
\]
We take for convenience the natural logarithm in the above expression.  The R\'enyi divergence interpolates several important statistical distances.  In particular the Kullback-Leibler divergence also referred to as the relative entropy,
\[
    D_1(X||Z) = \lim_{\alpha \to 1} D_\alpha(X||Z) = D_{KL}(X||Z) = \int_E \log \left( \frac{d\mu}{d\gamma} \right) d\gamma,
\]
the Hellinger distance
\[
    \mathcal{H}el^2(X;Z) = \int_E \left(\sqrt{ \frac{d\mu}{d\gamma}} - 1\right) d\gamma
\]
satisfies 
\[
D_{ 1 /2}(X||Z) = - 2 \log \left( 1 - \frac{ \mathcal{H}el^{2}(X;Z)}{2} \right),
\]
and the $\chi^2$ divergence defined by 
\[
    \chi^2(X;Z) = \int_E \left( \frac{d\mu}{d\gamma} -1 \right)^2 d\gamma,
\]
satisfies $D_2(X||Z) = \log \left( 1 + \chi^2(X;Z) \right)$.

The $\alpha$-R\'enyi divergence is a monotone function of an  $\alpha$-divergence, the $f$ divergence associated to the convex function 
\[
    f_\alpha(x) \coloneqq \frac{x^\alpha -1}{\alpha(\alpha-1)}.
\]
We direct the reader to \cite{Ren61, Csi63,morimoto1963markov, AS66} for the origins of the study of $f$-divergences, and to \cite{sason2016f, polyanskiy2025information} for background on inequalities between $f$-divergences.  We mention that comparisons between two $f$-divergences are completely understood, see \cite{harremoes2011pairs}.  However under restrictions of the class of variables, improvements of the general bounds are possible, for example \cite{melbourne2020strongly}.
The following result says that the distance to Gaussianity of a strongly log-concave variable in $\alpha$-R\'enyi divergence \cite{van2014renyi} is, up to a sharp constant, independent of $\alpha \in (0,\infty)$.  This is studied in applications as well.  The special case of Hellinger and KL divergences \cite{grunwald2020fast, wong1995probability}.  The general case is assumed \cite{khribch2024convergence}. 
 \begin{theorem}
  For $0<\alpha < \beta < \infty$, $X$ a random variable with a log-concave density with respect to a Gaussian measure $\gamma$, then $Z \sim \gamma$ implies
     \[
        D_\alpha(X||Z) \leq D_\beta(X||Z) \leq \frac{\beta}{\alpha} D_\alpha(X||Z).
     \]
     {\color{black}with equality in the second inequality iff $X$ is a translation of $Z$}
 \end{theorem}
We observe that the structure of the extremizers of the difference between $\alpha$ and $\beta$-R\'enyi divergences to a standard Gaussian among strongly log-concave random variables is much more rigid than its Euclidean counterpart.  Note that through a symmetry for $\alpha \in (0,1)$ one has $D_\alpha(X||Z) = \frac{\alpha}{1-\alpha} D_{1-\alpha}(X||Z)$ Thus for $0 <\alpha < \beta < 1$,
\begin{align*}
    D_\beta(X||Z) 
        &= 
            \frac{\beta}{1-\beta}D_{1-\beta}(Z||X)
                \\
        &\leq 
            \frac{\beta}{1-\beta}D_{1-\alpha}(Z||X)
                \\
        &=
            \frac{\beta}{\alpha}\frac{1-\alpha}{1-\beta} D_\alpha(X||Z).
\end{align*}
we see that all $\alpha,\beta \in (0,1)$ the R\'enyi divergences are equivalent in general.  However, this does not give any access to the KL-divergence case when $\beta =1$ as such divergences are in general not equivalent for $\beta \geq 1$.  
 \begin{proof}
     The first inequality is standard and an immediate consequence of the Jensen's inequality applied to the convexity of $x \mapsto x^{\frac \beta \alpha}.$  Now if $X \sim \mu$, is such that $\frac{d\mu}{d\gamma}$ is log-concave, then
     \[
        \frac{d\mu}{d\gamma}  = e^f
     \]
     for some concave function $f$.  By Corollary \ref{cor: convex function of Gaussian}  the map
     \[
        p \mapsto \frac 1p \log \int_E e^{pf}  d\gamma 
     \]
     is concave, and strictly so unless $f(Z)$ is a Gaussian, which by Theorem \ref{thm: convex functions preserving Gaussianity are affine} is strictly concave unless $f$ is an affine function. Moreover for $p \neq 1$ we have
     \[
     \frac 1p \log \int_E e^{pf}  d\gamma = \frac{p-1}{p} \frac{\log \int \left(\frac{d\mu}{d\gamma}\right)^{p} d\gamma}{p-1} = \frac{p-1}{p} D_p(X ||Z)
     \]
     Writing $\alpha = (1-t) + t \beta$ where $t = \frac{\alpha -1}{\beta -1}$, we have
     \[
        \frac 1 \alpha \log \int_E e^{\alpha f}  d\gamma \geq \frac{1-t}{1} \log  \int_E \frac{d\mu}{d\gamma} d\gamma + \frac{t}{\beta} \int_E \left( \frac{d\mu}{d\gamma} \right)^{\beta} d\gamma = \frac{t}{\beta} \int_E \left( \frac{d\mu}{d\gamma} \right)^{\beta} d\gamma.
     \]
     Rewriting this in terms of the R\'enyi divergence our result follows for $\alpha, \beta \neq 1$.  One can conclude for the $KL$-divergence through continuity.  For the equality case, as mentioned by Corollary \ref{cor: convex function of Gaussian} combined with Theorem \ref{thm: convex functions preserving Gaussianity are affine} we have equality if and only if $f$ is an affine function, say $f(x) = \langle a, x \rangle + b$ in which case writing $\frac{d\mu}{dx} = e^{-V}$
     \[
        \langle a, x \rangle + b = \log \frac{d\mu}{d\gamma} = {\frac{|x|^2}{2} - V}.
     \]
     Thus we have equality if and only if $V(x) = - \frac{|x|^2}{2} + \langle a, x \rangle - b$ which is if and only if $X$ is a translation of $Z$.
 \end{proof}
 
 Note that in particular we have
 \[
    D_{2}(X||Z) \leq 2 D(X||Z) \leq 4 D_{ 1 /2}( X||Z)
 \]
which is 
\[
       \log \left( 1 + \chi^2(X;Z) \right) \leq 2 D(X||Z) \leq  -  8 \log \left( 1 - \frac{ \mathcal{H}el^{2}(X;Z)}{2} \right) 
\]

 \section{Conic Intrinsic Volumes} \label{sec: conic intrinsic volumes}
 
The study of conic intrinsic volumes associated to convex cone, is a classical subject in convex geometry, going back at least to Sommerville \cite{sommerville1927relations}.  However, there has been recent surge of interest in the subject due to the concentration of such objects giving a rigorous explanation for the observed phase transitions in convex optimization problems, see \cite{amelunxen2014living}.  Further improvements of said concentration bounds can be found in \cite{mccoy2014steiner} and their typically Gaussian behavior in large dimension is proven in \cite{goldstein2017gaussian}. Our contribution below shows that conic intrinsic volume random variables are sub-Gaussian with a variance proxy determined in the lower tail by variance of the norm squared of a standard Gaussian projected to the cone, while its upper tail is controlled by the analogous construction projected to the dual cone.  This mirrors the sub-Gaussian concentration results achieved for intrinsic volumes associated to convex bodies \cite{LMNPT18,AMM22,marsiglietti2022moments,lotz2025sharp}

For $C$ a convex polyhedral cone in $\mathbb{R}^n$ and $\Pi_C$ the projection to said cone, let $F_i$ be the faces of $C$.  Then the sequence of (conic) intrinsic volumes are defined by
\[
    v_k \coloneqq v_k(C) \coloneqq \mathbb{P}(\Pi_C(Z) \in rel int (F_i), dim(F_i) = k )
\]
for $Z$ a standard $\mathbb{R}^n$ valued random vector.  For a general cone $C$ such that $C_n \to C$ in the ``conic Haussdorff metric'', then $v_k(C) \coloneqq \lim_{n \to \infty} v_k(C_n)$
we let $V_C$ denote a random variable such that $\mathbb{P}(V_C = k ) = v_k$.  We denote the dual cone associated to $C$ by $C^\circ \coloneqq \{ x \in \mathbb{R}^n : \langle x, y \rangle \leq 0,  \forall y \in C \}$

\begin{lemma}
    For $V_C$ a conic intrinsic volume random variable associated to a convex cone $C \subseteq \mathbb{R}^n$, $Z$ a standard $\mathbb{R}^n$ valued Gaussian random vector, y $\Pi_C(x) = \arg \min_{y \in C} | x - y | $, 
    \begin{align} \label{eq: statistical dimension}
        \delta(C) \coloneqq \mathbb{E} V_C = \mathbb{E} | \Pi_C(Z)|^2,
    \end{align}
    \[
        v_k(C^\circ) = v_{n-k}(C),
    \]
    in particular $\mathbb{E} V_C = n - \mathbb{E}V_{C^\circ}$
    \[
        \var(V_C) = \var( |\Pi_C(Z)|^2 ) - 2 \delta(C) = \var( | \Pi_{C^{\circ}}(Z)|^2) - 2 \delta( C^\circ) = \var(V_{C^\circ})
    \]
    for $\eta \in \mathbb{R}$ y $\xi = \frac 1 2 (1 - e^{-2 \eta})$, so that $\eta = - \frac 1 2 \log ( 1 - 2 \xi)$ we have
    \begin{align} \label{eq: MGF representation}
        \mathbb{E} e^{\eta V_C} = \mathbb{E} e^{\xi | \Pi_C(Z)|^2}.
    \end{align}
   Additionally,  each $x\in \mathbb{R}^n$ can be  decomposed as an orthogonal sum of  the two projections to $C$ and $C^\circ$
    \begin{align} \label{eq: decomp Polars}
         \Pi_C(x) + \Pi_{C^\circ}(x)  = x
    \end{align}
\end{lemma}

\begin{lemma} \label{lem: norm of projection to convex cone is convex}
    For $C$ a convex cone, and $\Pi_C(x) = \arg \min_{y \in C} |x - y|$,
    \[
        x \mapsto |\Pi_C(x) |^2
    \]
    is a convex function.
\end{lemma}
\begin{proof}
    Taking $t \in (0,1)$ and $x,y \in \mathbb{R}^n$, by \eqref{eq: decomp Polars},
    \begin{align*}
        | \Pi_C ((1-t) x + ty ) |^2
            &=
                | (1-t) x + t y - \Pi_{C^\circ}((1-t) x + ty )|^2.
    \end{align*}
    By the convexity of $C^\circ$, $(1-t) \Pi_{C^\circ}(x) + \Pi_{C^\circ}(y) \in C^\circ$, and hence by the definition of the projection,
    \begin{align*}
        | (1-t) x + t y - \Pi_{C^\circ}((1-t) x + ty |^2 
            &\leq
                | (1-t) x + t y - ((1-t) \Pi_{C^\circ}(x) + t\Pi_{C^\circ}(y)) |^2
                    \\
            &=
                | (1-t) (x-\Pi_{C^\circ}(x)) + t (y -\Pi_{C^\circ}(y)) |^2
                    \\
            &\leq
                (1-t) | x - \Pi_{C^\circ}(x) |^2 + t |y - \Pi_{C^\circ}(y) |^2
                    \\
            &=
                (1-t) |\Pi_{C}(x) |^2 + t | \Pi_{C}(y) |^2
    \end{align*}
\end{proof}
\begin{theo}
For a convex cone $C$ and $V_C$ a corresponding conic intrinsic volume random variable with statistical dimension $\delta(C) = \mathbb{E} V_C$, and $ t \leq  \delta(C),$
\[
    \mathbb{P}( V_C \leq \delta(C) - t) \leq \exp \left\{- \frac{t^2}{2\sigma^2} \right\}
\]
where $\sigma^2 = \var( | \Pi_C |^2) = \var(V_C) + 2 \delta(C)$
\end{theo}

\begin{proof}
    For $\eta < 0$, applying Markov's inequality,
    \begin{align*}
        \mathbb{P}(V_C \leq \delta(C) - t)
            &=
                \mathbb{P}( e^{\eta V_C} \geq e^{ \eta(\delta(C) -t) })
                    \\
            &\leq
                \mathbb{E} e^{\eta V_C} e^{\eta(t -\delta(C) )}
                    \\
            &=
                \mathbb{E} e^{ \xi | \Pi_C(Z) |^2}e^{\eta (t -\delta(C) )}
    \end{align*}
    with $\xi = \frac 1 2 (1 - e^{-2 \eta} )$ so that $\eta = - \frac{1}{2}\log(1 - 2 \xi)$ coming from \eqref{eq: MGF representation}.  Note that $\xi \leq 0$ so that using the covexity of $|\Pi_C(x) |^2$  from \eqref{lem: norm of projection to convex cone is convex}
    to apply \eqref{eq: MGF bound under convexity},
    we have
    \begin{align*}
        \mathbb{P}(V_C \leq \delta(C) - t) &\leq \exp \left\{ \frac{\xi^2}{2} \var( | \Pi_C(Z) |^2) + \xi \mathbb{E} |\Pi_C(Z)|^2 + \frac{\log \left( 1 - 2 \xi \right)}{2} (\delta(C) - t) \right\}
            \\
        &\leq \exp \left\{ \frac{\xi^2}{2} \var( | \Pi_C(Z) |^2) + \xi \delta(C) -  \xi (\delta(C) - t) \right\}
            \\
             &= \exp \left\{ \frac{\xi^2}{2} \var( | \Pi_C(Z) |^2) +\xi t \right\}
    \end{align*}
    where the second inequality is from $\log (1 + x) \leq x$ and the equality \eqref{eq: statistical dimension}, and then algebra in the last line.
    Taking $\xi = - \frac{t}{\var( |\Pi_C |^2) }$
    we arrive at 
    \[
        \mathbb{P}(V_C \leq \delta(C) - t) \leq \exp \left\{ \frac{- t^2}{2\var( | \Pi_C(Z) |^2)}   \right\}
    \]
\end{proof}

For a converse inequality,
\[
    \mathbb{P}(V_C \geq \delta(C) + t) = \mathbb{P}(n - V_{C^\circ} \geq n -\delta(C^\circ ) + t ) = \mathbb{P}(V_{C^\circ} \leq \delta(C^\circ) - t) \leq e^{- \frac{t^2}{2 \var(| \Pi_{C^\circ}(Z) |^2} }
\]
Note that $\var (| \Pi_{C^\circ}(Z) |^2) = \var(V_C) + 2 (n - \delta(C))$

Note for example if $V_C$ is symmetric so that $\delta(C) = \frac n 2$ we have
\[
    \mathbb{P}( |V_C - \delta(C)| \geq t) \leq 2 \exp \left\{ - \frac{t^2}{\var(V_C) + n } \right\} 
\]
The following states that all conic intrinsic volume random variables are sub-Gaussian.
\begin{coro}
    For $C$ a convex cone and $V_C$ the associated intrinsic volume random variable,
    \[
        \mathbb{P}(|V_C - \delta(C)| \geq t ) \leq 2 e^{- \frac{t^2}{2 \sigma^2}}
    \]
    with sub-Gaussian variance proxy $\sigma^2 \coloneqq\var(V_C) +  2 \max \{ \delta(C), n - \delta(C)\}$.
\end{coro}

We note that unlike the Euclidean intrinsic volumes where the Alexandrov-Fenchel inequality delivers immediate concavity properties, and a host of concentration, moment, and entropy bounds can be immediately obtained via general inequalities for log-concave sequences, see for instance, \cite{BMM20,alqasem2024conjecture, jakimiuk2024log, gaxiola2025anti}, it is not known if, but conjectured in \cite{amelunxen2011geometric} that the conic intrinsic volumes are log-concave.  Regarding the concentration inequalities obtained here, an improvement of constants based on log-concavity properties, would follow if the conic intrinsic volume sequences could for instance, be proven to be ultra log-concave of finite order, see \cite{pemantle2000towards, marsiglietti2022moments}.

\section{A Wills Functional Inequality} \label{sec: Wills}
For a convex set $K \subseteq \mathbb{R}^n$ and $\lambda \geq 0$, the Steiner formula states that the expansion of $K$ by $\lambda$ dilations of $B = \{ x \in \mathbb{R}^n : |x|_2 \leq 1 \}$ is a polynomial in $\lambda$,
\[
    \vol(K + \lambda B) = \sum_{i = 0}^n \binom{n}{i} W_i(A) \lambda^i.
\]
The coefficients $W_i(A)$ are a special case of quermassintegrals which play a central role in convex geometry \cite{schneider2014convex}, however they are dependent on the dimension of the Euclidean space within which $A$ is embedded.  This was addressed in McMullen \cite{mcmullen1975non} who introduced the intrinsic volumes associated to a convex  set $A$, defined by
\[
    V_i(K) \coloneqq \frac{\binom{n}{i}}{\omega_{n-i}} W_{n-i}(K)
\]
for $i \in \{0,1,\dots, n\}$, where $\omega_k$ is the volume of the $k$-dimensional unit ball, explicitly,
\[
\omega_k = \frac{\pi^{\frac k 2}}{\Gamma\left(\frac k 2 + 1 \right)}.
\]
Of particular importance for our investigations here, is $V_1(K)$ which has the following representation, see \cite{schneider2014convex},
\begin{align} \label{eq: V1 representation}
    V_1(K) = \frac{1}{\omega_{n-1}} \int_{\mathbb{S}^{n-1}} h_K(\theta) d\sigma(\theta)
\end{align}
where $h_K$ denotes the support function of $K$, $h_K(\theta) = \sup \{ \langle \theta, y\rangle : y \in K\}$ and $\sigma$ denotes the Lebesgue measure on the unit sphere $\mathbb{S}^{n-1}$.
The sum of all intrinsic volumes is called the Wills functional \cite{wills1973gitterpunktanzahl},
\[
    W(K) \coloneqq \sum_{i=0}^n V_i(K).
\]
In \cite{Had79}, Hadwiger showed that the Wills functional had the following equivalent form, which we take as its definition.
\begin{defn}
    For convex $K \subseteq \mathbb{R}^n$,
    \[
        W(K) \coloneqq \int_{\mathbb{R}^n} e^{- \pi d^2(x,K)} dx
    \]
    where
    \[
        d(x,K) \coloneqq \inf_{y \in K} |x-y|.
    \]
\end{defn}
More background on intrinsic volumes can be found in \cite{Sch14:book}, the reader can consult \cite{alonso2021further} for more on the Wills functional.  
Following Vitale \cite{vitale1996wills}, we observe that the subsitution $w = \sqrt{2 \pi} \ x $ gives 
\begin{align*}
    \int_{\mathbb{R}^n} e^{- \pi d^2(x,K)} dx
        &=
            \int_{\mathbb{R}^n} \exp \left\{ \sup_{y \in K} \left( - \pi | x - y|^2 \right) \right\} dx
                \\
        &=
            \int_{\mathbb{R}^n} \exp \left\{ 2\pi  \sup_{y \in K} \left(   \langle x, y \rangle  - \frac{|y|^2}{2} \right) \right\} e^{- \pi |x|^2 }dx
                \\
        &=
            \int_{\mathbb{R}^n}\exp \left\{  \sup_{y \in K} \left(   \langle x, \sqrt{2 \pi }\  y \rangle  - \frac{|\sqrt{2 \pi } \ y|^2}{2} \right) \right\} e^{-  \frac{|x|^2}{2} }dx
                \\
        &=
            \mathbb{E} \exp \left \{ \sup_{y \in \sqrt{2 \pi} \ K} \left( \langle Z, y \rangle - \frac{ |y|^2}{2} \right)  \right\},
\end{align*}
where $Z \sim N(0, I_n)$.
Summarizing,
\[
    W \left( {K} \right) = \mathbb{E} e^{f(Z)}
\]
where $f \coloneqq f_K$ is the suprema of the affine functions, and hence a convex function defined by
\[
    f(z) \coloneqq \sup_{y \in \sqrt{2 \pi} K} \left( \langle z , y \rangle - \frac{|y|^2}{2} \right)
\]

\begin{prop} \label{prop: V1 Gaussian}
    For a convex set $K$,
    \[
        V_1(K) = \mathbb{E} \left[ \sup_{y \in \sqrt{2 \pi} \ K} \langle Z, y \rangle \right]
    \]
    for $Z \sim N(0, I_n)$.
\end{prop}

\begin{proof}
    We will compute directly, using the fact that for $Z$ a standard Gaussian vector, $|Z|$ and $Z/|Z|$ are independent with $|Z|$ $\chi$-distributed of order $n$, and $Z/|Z|$ is uniformly distributed on $\mathbb{S}^{n-1}$.
    \begin{align*}
         \mathbb{E} \left[ \sup_{y \in \sqrt{2 \pi} \ K} \langle Z, y \rangle \right]
            &=
                \sqrt{2 \pi} \mathbb{E} \left[ |Z| \sup_{y \in K} \langle Z/|Z|, y \rangle \right]
                    \\
            &=
                \sqrt{2 \pi} \frac{\mathbb{E}|Z|}{\vol_{n-1}(\mathbb{S}^{n-1})} \int_{\mathbb{S}^{n-1}} h_K(\theta) d\sigma(\theta) 
                    \\
            &=
                \sqrt{2 \pi} \frac{\mathbb{E}|Z| \omega_{n-1}}{\vol_{n-1}(\mathbb{S}^{n-1})} V_1(K),
    \end{align*}
    where we have used \eqref{eq: V1 representation} in the last equality.  Note that $\vol_{n-1}(\mathbb{S}^{n-1}) = n \omega_n$ and $\mathbb{E}|Z| = \sqrt{2} \frac{\Gamma((n+1)/2)}{\Gamma(n/2)}$, using $\frac{n}{2} \Gamma(n/2) = \Gamma(1 + n/2)$ we see that the constants cancel and the proof follows.
\end{proof}


McMullen \cite[Theorem 2.2]{mcmullen1991inequalities} proved that 
\[
    W(K) \leq e^{V_1(K)}.
\]
In what follows we pursue a reversal of this inequality due to \cite{alonso2021further}.  We mention that $V_1(K)$ is, up to a dimensionally dependent constant, equivalent to the Gaussian mean width.  We direct the  interested reader to background and applications of the mean width in \cite{vershynin2015estimation}.
\begin{theo}
    For $K$ a convex body, $f_K(z) = \sup_{y \in \sqrt{2 \pi} K} \left( \langle z , y \rangle - \frac{|y|^2}{2} \right)$
    \[
        \log W(K) \geq \frac{\var( f_K(Z))}{2} + \mathbb{E} f_K(Z)
    \]
\end{theo}

\begin{proof}
    Taking $X = f_K(Z) - \mathbb{E}f_K(Z)$, by Theorem \ref{thm: main generalized} the map $\Lambda_X(\lambda) = \frac 1 \lambda \log \mathbb{E} e^{\lambda X}$ is convex, and satisfies $\Lambda_X(0) = 0$ and $\Lambda_X'(0) = \frac{\var(f_K(Z))}{2}$.  Thus,
    \[
        \Lambda_X(1) = \frac{\Lambda_X(1) - \Lambda(0)}{1- 0} \geq \Lambda_X'(0) = \frac{\var(f_K(Z))}{2}
    \]
    Observing that 
    \[
        \Lambda_X(1) = \log \mathbb{E}e^{f_K(Z) - \mathbb{E}f_K(Z)} = \log W(K) - \mathbb{E} f_K(Z),
    \]
    completes the proof.
\end{proof}
As a corollary we obtain Theorem \ref{thm: intro reverse mcmullen}.
\begin{coro}
    For $K \subseteq r B$, where $B$ is the Euclidean  unit ball,
    \[
        \log W(K) \geq \frac{\var( f_K(Z))}{2} +V_1(K) - \frac{r^2}{2}.
    \]
\end{coro}
Thus, by the non-negativity of the variance, we recover a strengthening of \cite[Theorem 1.2]{alonso2021further} where the inequality
\[
     \log W(K) \geq V_1(K) - \frac{r^2}{2}
\]
is obtained.  We direct the reader to \cite{mourtada2025universal} for recent applications of this inequality.

\begin{proof}
    Since $y\in K$ implies $|y| \leq r$, applying Proposition \ref{prop: V1 Gaussian},
    \[
        \mathbb{E} f_K(Z)  = \mathbb{E} \sup_{y \in \sqrt{2 \pi} K} \left( \langle Z, y \rangle - \frac{|y|^2}{2} \right) \geq \mathbb{E} \sup_{y\in \sqrt{2 \pi} K} \langle Z, y\rangle - \frac{r^2}{2} = V_1(K) - \frac{r^2}{2},
    \]
    completes the proof.
\end{proof}


\appendix

\section{An Algebraic Lemma}
\begin{lemma}\label{lem: exponential lin ind}
Given $p_0, \dots, p_n \in \mathbb{R}$ with $p_i < p_{i+1}$ and $n+1$ distinct points $t_0, \dots, t_n$, the vectors $a_i = \begin{pmatrix}
e^{p_i t_0} \\
\vdots \\
e^{p_i t_n}
\end{pmatrix}$ are linearly independent.  Equivalently, any function of the form 
\[
    \Psi_c(x) =  \sum_{i=0}^n c_i e^{p_i x}
\]
for $c = (c_i) \neq 0$, has at most $n$ distinct zeros.
\end{lemma}
\begin{proof}
To see the equivalence, if there exists $c \neq 0$ such that $\Phi_c$ has more than $n$ zeros, say $t_1, \dots,t_{n+1}$, Then $\sum_{i=1}^n c_i a_i = 0$, and the $a_i$ are linearly dependent.  Conversely, if $t_0, \dots, t_n$ are distinct points such that the vectors $a_i$ linearly dependent then there exists $c$ such that $\sum_{i=0}^n c_i a_i = 0$, using this $c$ to define $\Phi_c$, we have a sum of $n+1$ exponentials with $t_0, \dots, t_n$ as zeros.
    We now proceed by induction to show that a linear combination of $n+1$ exponential functions has no more than $n$ zeros. When $n=0$, $e^{p_0 x}$ has no zeros and the claim holds. Suppose the claim for $n - 1 \in \mathbb{N}$ and that $\Phi_c(x) = \sum_{i=0}^n c_i e^{p_i x}$ has distinct zeros $t_0, \dots, t_{m-1}$. Multiplying by $e^{-p_0 x}$, the function 
    $$\tilde{\Psi}_c(x) = e^{-p_0 x} \Psi_c(x) =  \sum_{i = 0}^n c_i e^{(p_i - p_0) t_k} $$ satisfies $\tilde{\Psi}_c(t_i) = 0$ and its derivative
    \[
        \tilde{\Psi}_c'(x) = \sum_{i =1}^n c_i (p_i - p_0) e^{p_i x}
    \]
    has, by Rolle's theorem, a  zero on every interval $(t_i, t_{i+1})$.  Thus $\tilde{\Psi}'$ has at least $m-1$ zeros, but by induction hypothesis, since $\tilde{\Psi}'$ is a non-trivial linear combination of $n$ distinct exponential functions, it has no more than $n-1$ zero.  Thus $m -1 \leq n-1$ and the result follows. 
\end{proof}

\begin{lemma}\label{lem: exponential algebra stuff}
    For $p_0 < \dots < p_n$ and $c = (c_0, \dots, c_n) \neq \mathbf{0}$, if the function
    \[
    \Psi_c (x) = \sum_{i = 0}^n c_i e^{p_i x}
    \]
    has exactly $n$ zeros $t_0, \dots, t_{n-1}$, then $c_i c_{i+1} < 0$ and there exists an $\varepsilon > 0$ such that $s \in (0,\varepsilon)$ implies $\Psi_c (t_i + s) \Psi_c (t_i - s) < 0$.
\end{lemma}

\begin{proof}
     We proceed again by induction on $n$. For $n=0$, $x \mapsto ce^{px}$ has no zeros, so our lemma is true vacuously, the proof when $n=1$, $\Psi_c$ has a unique zero at $t = \frac{\log \left( - \frac{c_0}{c_1} \right)}{p_1-p_0}$ if and only if $c_0 c_1 < 0$ from which the claim follows. Thus we suppose $n \geq 2$ and that our claim is true for $n-1 \in \mathbb{N}$ and that $\Psi_c(x) = \sum_{i=0}^n c_i e^{p_i x}$  has $t_1 < \cdots < t_n$ zeros.  The function
     \[
        \tilde{\Psi}_c(x) = e^{-p_k x} \Psi_c(x) = \sum_{i=0}^n c_i e^{(p_i - p_k)x}
     \]
     has $t_0, \dots, t_{n-1}$ as zeros as well.  Moreover, for $x \notin \{t_0,\dots, t_{n-1}\}$ we have $\frac{\tilde{\Psi}_c}{\Psi_c} > 0$ and taking the a derivative, 
     \[
        \tilde{\Psi}_c'(x) = \sum_{i \neq k } c_i (p_i - p_k) e^{(p_i - p_k)x}
     \]
     is a sum of $n$ exponential functions. By Lemma \ref{lem: exponential lin ind} has at most $n-1$ zeros.  However, by Rolle's theorem $\tilde{\Phi}_c'$ has zeros $s_i \in (t_i, t_{i+1})$ for $0 \leq i \leq n-2$ and hence $\tilde{\Psi}_c'$ has exactly $n-1$ zeros. By the inductive hypothesis we can conclude.  When $k = 0$, we have
     \[
        c_i c_{i+1} (p_i - p_0) (p_{i+1} - p_0) < 0.
     \]
     Since the $p_i$'s are ordered, $(p_i - p_0) (p_{i+1} - p_0) > 0$ and we have $c_i c_{i+1} < 0$ for $i \geq 1$.  Similarly for $k = n$,
     \[
         c_i c_{i+1} (p_i - p_k) (p_{i+1} - p_k)
     \]
    and $c_i c_{i+1} < 0$ for $i \leq n-2$.  The interlacing of the zeros of $\tilde{\Psi}_c'$ and $\tilde{\Psi}_c$ forces $\tilde{\Psi}_c$ and hence $\Psi_c$ as well to switch signs at each $t_i$, completing the proof.
\end{proof}

\section*{Acknowledgements}
The second and third author's work was supported by CONAHCYT grant CBF2023-2024-3907.  We thank Arturo Jaramillo, Mokshay Madiman, and Shuya Yu for helpful discussions.
\bibliographystyle{plain}
\bibliography{bibibi}

\end{document}